\documentclass[a4paper]{amsart}
\usepackage{amscd}
\usepackage{amsmath}
\usepackage{amssymb}
\usepackage{amsthm}
\usepackage{bbm} 
\usepackage{mathrsfs}
\usepackage{stmaryrd}
\usepackage{tikz-cd}
\usepackage[arrow, matrix, curve]{xy}

\usepackage[T1]{fontenc}
\usepackage[utf8]{inputenc}

\usepackage{lipsum}

\makeatletter
\renewcommand\subsection{\@startsection{subsection}{2}%
  \z@{-.5\linespacing\@plus-.7\linespacing}{.5\linespacing}%
  {\normalfont\scshape}}
\renewcommand\subsubsection{\@startsection{subsubsection}{3}%
  \z@{.5\linespacing\@plus.7\linespacing}{-.5em}%
  {\normalfont\scshape}}
\makeatother

\usepackage[T1]{fontenc}

\makeatletter
\@namedef{subjclassname@2020}{%
  \textup{2020} Mathematics Subject Classification}
\makeatother

\numberwithin{equation}{section} \swapnumbers

\newtheorem{satz}{Satz}[section]

\newtheorem{theorem}[satz]{Theorem}
\newtheorem{proposition}[satz]{Proposition}
\newtheorem{corollary}[satz]{Corollary}
\newtheorem{lemma}[satz]{Lemma}

\newtheorem{definition}[satz]{Definition}

\newtheorem{remark}[satz]{Remark}

\newcommand{\bbr}{\mathbb{R}}

\newcommand{\bbn}{\mathbb{N}}
\newcommand{\bbp}{\mathbb{P}}

\newcommand{\bbb}{\mathbb{B}}

\newcommand{\calb}{\mathcal{B}}

\newcommand{\cale}{\mathcal{E}}
\newcommand{\calf}{\mathscr{F}}

\newcommand{\calm}{\mathcal{M}}
\newcommand{\caln}{\mathcal{N}}

\newcommand{\calp}{\mathcal{P}}
\newcommand{\cals}{\mathscr{S}}

\newcommand{\Id}{{\rm Id}}

\newcommand{\lin}{{\rm lin}}

\newcommand{\bbI}{\mathbbm{1}}

\begin{document}

\hyphenation{rea-li-za-tion pa-ra-me-tri-za-ti-ons pa-ra-me-tri-za-ti-on Schau-der sub-ma-ni-fold sub-ma-ni-folds}

\title[Invariant manifolds for stochastic partial differential equations]{A note on invariant manifolds for stochastic partial differential equations in the framework of the variational approach}
\author{Rajeev Bhaskaran \and Stefan Tappe}
\address{Indian Institute of Science Education and Research Thiruvanantapuram, Department of Mathematics, Maruthamala P.O, 695551,
Kerala, India}
\email{brajeev@iisertvm.ac.in}
\address{Albert Ludwig University of Freiburg, Department of Mathematical Stochastics, Ernst-Zermelo-Stra\ss{}e 1, D-79104 Freiburg, Germany}
\email{stefan.tappe@math.uni-freiburg.de}
\date{2 October, 2024}
\thanks{Stefan Tappe gratefully acknowledges financial support from the Deutsche Forschungsgemeinschaft (DFG, German Research Foundation) -- project number 444121509.}
\begin{abstract}
In this note we provide conditions for local invariance of finite dimensional submanifolds for solutions to stochastic partial differential equations (SPDEs) in the framework of the variational approach. For this purpose, we provide a connection to SPDEs in continuously embedded spaces.
\end{abstract}
\keywords{Stochastic partial differential equation, variational approach, continuously embedded Banach spaces, invariant manifold}
\subjclass[2020]{60H15, 60H10, 60G17}

\maketitle\thispagestyle{empty}

\section{Introduction}

In this paper we deal with invariant manifolds for stochastic partial differential equations (SPDEs) in the framework of the variational approach. More precisely, let $G,K$ be Banach spaces, and let $H$ be a separable Hilbert space such that $(G,H,K)$ is a triple of continuously embedded spaces. Consider a $(G,H,K)$-variational SPDE of the form
\begin{align}\label{SPDE}
\left\{
\begin{array}{rcl}
dY_t & = & L(Y_t) dt + A(Y_t) dW_t
\\ Y_0 & = & y_0
\end{array}
\right.
\end{align}
driven by an $\bbr^{\infty}$-Wiener process $W$ with measurable coefficients $L : G \to K$ and $A : G \to \ell^2(H)$. Variational SPDEs of this kind have been studied, for example, in \cite{Rozovskii, Prevot-Roeckner, Liu-Roeckner}. Given a finite dimensional submanifold $\calm$, we are interested in the question when $\calm$ is locally invariant for the SPDE (\ref{SPDE}); a property, which in particular ensures the existence of local solutions.

Usually, when dealing with variational SPDEs, a series of conditions on the coefficients $L$ and $A$ is required in order to ensure existence and uniqueness of solutions. These conditions are hemicontinuity, monotonicity, coercivity and growth conditions; see, e.g., conditions (H1)--(H4) in \cite[Chap. 4]{Liu-Roeckner} and conditions (H1), (H2'), (H3), (H4') in \cite[Chap. 5]{Liu-Roeckner}. As we will show, continuity of the coefficients is sufficient for the existence of solutions, provided that a finite dimensional invariant submanifold exists; see Theorems \ref{thm-K-separable} and \ref{thm-general}.

The invariance problem for submanifolds has been investigated, for example, in \cite{Milian-manifold} for finite dimensional SDEs, in \cite{Filipovic-inv, Nakayama} for SPDEs in the framework of the semigroup approach, and recently in \cite{BT} for SPDEs in continuously embedded spaces. The latter framework is similar to the framework of the variational approach, and this is a crucial observation for providing our results from this paper. Indeed, if $L$ is even a measurable mapping $L : G \to H$ with values in the Hilbert space $H$, then we may regard (\ref{SPDE}) also as a $(G,H)$-embedded SPDE in the spirit of \cite{BT}. This allows to reduce a problem involving three spaces to a problem involving only two spaces.

The remainder of this paper is organized as follows. In Section \ref{sec-SPDEs-embedded} we provide the required background about SPDEs in continuously embedded spaces, and in Section \ref{sec-SPDEs} we provide the required background about SPDEs in the framework of the variational approach. After these preparations, we present our results on invariant manifolds; namely in Section \ref{sec-K-separable} we consider the situation where $K$ is a separable Hilbert space, and in Section \ref{sec-general} we consider the general situation. Afterwards, we use our findings for the following two applications. In Section \ref{sec-HS} we construct examples of invariant submanifolds in Hermite Sobolev spaces, and in Section \ref{sec-Laplace} we characterize linear submanifolds for the stochastic $p$-Laplace equation. For convenience of the reader, in Appendix \ref{app-embedding} we provide the required auxiliary results about continuously embedded spaces.

\section{Stochastic partial differential equations in continuously embedded spaces}\label{sec-SPDEs-embedded}

In this section we provide the required prerequisites about SPDEs in continuously embedded spaces. It is similar to \cite[Sec. 2]{BT}, where further details can be found.

Let $G$ be a Banach space and let $H$ be a separable Hilbert space such that $(G,H)$ is a pair of continuously embedded spaces; see Definition \ref{def-embedding}. Let $\caln \subset G$ be a subset, endowed with the relative topology induced by $G$. In particular, we could choose $\caln := G$. Furthermore, let $L : \caln \to H$ and $A : \caln \to \ell^2(H)$ be measurable coefficients.

\begin{definition}\label{def-martingale-solution}
Let $y_0 \in \caln$ be arbitrary. A triplet $(\bbb,W,Y)$ is called a \emph{local martingale solution} to the $(G,H)$-embedded SPDE (\ref{SPDE}) with $Y_0 = y_0$ if the following conditions are fulfilled:
\begin{enumerate}
\item $\bbb = (\Omega,\calf,(\calf_t)_{t \in \bbr_+},\bbp)$ is a stochastic basis; that is, a filtered probability space satisfying the usual conditions.

\item $W$ is a standard $\bbr^{\infty}$-Wiener process on the stochastic basis $\bbb$.

\item $Y$ is an $\caln$-valued $\calp$-$\calb(\caln)$-measurable process such that for some strictly positive stopping time $\tau > 0$ we have $\bbp$-almost surely
\begin{align}\label{SPDE-int-cond}
\int_0^{t \wedge \tau} \big( \| L(Y_s) \|_H + \| A(Y_s) \|_{\ell^2(H)}^2 \big) ds < \infty, \quad t \in \bbr_+
\end{align}
and $\bbp$-almost surely
\begin{align}\label{SPDE-integral-form}
Y_{t \wedge \tau} = y_0 + \text{{\rm ($H$-)}} \int_0^{t \wedge \tau} L(Y_s) ds + \text{{\rm ($H$-)}} \int_0^{t \wedge \tau} A(Y_s) dW_s, \quad t \in \bbr_+.
\end{align}
The stopping time $\tau$ is also called the \emph{lifetime} of $Y$.
\end{enumerate}
\end{definition}

\begin{remark}
As usual $\calp$ denotes the predictable $\sigma$-algebra on $\Omega \times \bbr_+$. Moreover, as the notation indicates, the stochastic integral
\begin{align*}
\text{{\rm ($H$-)}} \int_0^{t \wedge \tau} L(Y_s) ds
\end{align*}
appearing in (\ref{SPDE-integral-form}) is a pathwise Bochner integral in the separable Hilbert space $(H,\| \cdot \|_H)$, and the stochastic integral
\begin{align*}
\text{{\rm ($H$-)}} \int_0^{t \wedge \tau} A(Y_s) dW_s
\end{align*}
appearing in (\ref{SPDE-integral-form}) is an It\^{o} integral in the separable Hilbert space $(H,\| \cdot \|_H)$.
\end{remark}

\begin{remark}
If there is no ambiguity, we will simply call $Y$ a local martingale solution to the $(G,H)$-embedded SPDE (\ref{SPDE}) with $Y_0 = y_0$.
\end{remark}

\begin{remark}\label{rem-cont-embedded}
Note that every local martingale solution $Y$ to the $(G,H)$-embedded SPDE (\ref{SPDE}) with lifetime $\tau$ is a $G$-valued process, and that by (\ref{SPDE-integral-form}) the sample paths of $Y^{\tau}$ are continuous with respect to $\| \cdot \|_H$.
\end{remark}

\begin{definition}\label{def-local-invariance}
A subset $\calm \subset \caln$ is called \emph{locally invariant} for the $(G,H)$-embedded SPDE (\ref{SPDE}) if for each $y_0 \in \calm$ there exists a local martingale solution $Y$ to the $(G,H)$-embedded SPDE (\ref{SPDE}) with $Y_0 = y_0$ and lifetime $\tau$ such that $Y^{\tau} \in \calm$ up to an evanescent set.
\end{definition}

Finite dimensional \emph{submanifolds} in embedded Banach spaces can be defined as in \cite[Sec. 3.1]{BT-preprint}, where Hilbert spaces have been considered. By virtue of the results from \cite[Sec. 6.1]{fillnm}, where finite dimensional submanifolds in Banach spaces have been studied, all definitions and results from \cite[Sec. 3.1]{BT-preprint} transfer to the present situation with Banach spaces. Let $k \in \bbn \cup \{ \infty \}$ be arbitrary. We recall that a finite dimensional $C^k$-submanifold $\calm$ of $H$ is called a $(G,H)$-submanifold of class $C^k$ if $\calm \subset G$ and $\tau_H \cap \calm = \tau_G \cap \calm$, where $\tau_G$ and $\tau_H$ denote the respective topologies.

\begin{proposition}\label{prop-embedding}\cite[Prop. 3.21]{BT-preprint}
Let $\calm$ be an finite dimensional $C^k$-submanifold of $H$. Then the following statements are equivalent:
\begin{enumerate}
\item[(i)] $\calm$ is a $(G,H)$-submanifold of class $C^k$.

\item[(ii)] $\calm \subset G$ and the identity $\Id : (\calm, \| \cdot \|_H) \to (\calm, \| \cdot \|_G)$ is continuous.
\end{enumerate}
\end{proposition}

Now, let us recall the following invariance result.

\begin{theorem}\label{thm-SPDE}\cite[Thm. 3.4]{BT}
Let $\calm$ be a finite dimensional $(G,H)$-submanifold of class $C^2$. Suppose that the mappings $L|_{\calm} : (\calm,\| \cdot \|_G) \to H$ and $A|_{\calm} : (\calm,\| \cdot \|_G) \to \ell^2(H)$ are continuous. Then the following statements are equivalent:
\begin{enumerate}
\item[(i)] The submanifold $\calm$ is locally invariant for the $(G,H)$-embedded SPDE (\ref{SPDE}).

\item[(ii)] We have
\begin{align}\label{tang-A}
&A^j|_{\calm} \in \Gamma(T \calm), \quad j \in \bbn,
\\ \label{tang-L} &[ L|_{\calm} ]_{\Gamma(T \calm)} - \frac{1}{2} \sum_{j=1}^{\infty} [A^j|_{\calm}, A^j|_{\calm}]_{\calm} = [0]_{\Gamma(T \calm)}.
\end{align}
\end{enumerate}
\end{theorem}

Here $\Gamma(T \calm)$ denotes the space of all vector fields on $\calm$; that is, the space of all mappings $A : \calm \to H$ such that $A(y) \in T_y \calm$ for each $y \in \calm$, where $T_y \calm$ denotes the tangent space to $\calm$ at $y$. Moreover, denoting by $A(\calm)$ the linear space of all mappings $A : \calm \to H$, and by $A(\calm) / \Gamma(T \calm)$ the quotient space, for two vector fields $A,B \in \Gamma(T \calm)$ the mapping $[A,B]_{\calm} \in A(\calm) / \Gamma(T \calm)$ is defined as follows. For each local parametrization $\phi : V \to U \cap \calm$ a local representative of $[A,B]_{\calm}$ on $U \cap \calm$ is given by
\begin{align}\label{bracket}
y \mapsto D^2 \phi(x) \big( D \phi(x)^{-1} A(y), D \phi(x)^{-1} B(y) \big), \quad y \in U \cap \calm,
\end{align}
where $x := \phi^{-1}(y) \in V$. For further details we refer to \cite[Sec. 3]{BT}.

\section{Stochastic partial differential equations in the framework of the variational approach}\label{sec-SPDEs}

In this section we provide the required prerequisites about SPDEs in the framework of the variational approach. Let $G,K$ be Banach spaces and let $H$ be a separable Hilbert space such that $(G,H,K)$ is a triplet of continuously embedded spaces. Let $\caln \subset G$ be a subset, endowed with the relative topology induced by $G$. Furthermore, let $L : \caln \to K$ and $A : \caln \to \ell^2(H)$ be measurable coefficients.

\begin{remark}
Note the difference to the framework from Section \ref{sec-SPDEs-embedded}, where we had a drift $L : \caln \to H$ with values in the Hilbert space $H$.
\end{remark}

\begin{definition}\label{def-martingale-solution-var}
Let $y_0 \in \caln$ be arbitrary. A triplet $(\bbb,W,Y)$ is called a \emph{local martingale solution} to the $(G,H,K)$-variational SPDE (\ref{SPDE}) with $Y_0 = y_0$ if the following conditions are fulfilled:
\begin{enumerate}
\item $\bbb = (\Omega,\calf,(\calf_t)_{t \in \bbr_+},\bbp)$ is a stochastic basis.

\item $W$ is a standard $\bbr^{\infty}$-Wiener process on the stochastic basis $\bbb$.

\item $Y$ is an $\caln$-valued $\calp$-$\calb(\caln)$-measurable process such that for some strictly positive stopping time $\tau > 0$ the image
\begin{align*}
\{ L(Y_t^{\tau}) : t \in \bbr_+ \} \subset K
\end{align*}
is $\bbp$-almost surely a separable subset of $(K,\| \cdot \|_K)$, we have $\bbp$-almost surely
\begin{align}\label{int-cond-var}
\int_0^{t \wedge \tau} \big( \| L(Y_s) \|_K + \| A(Y_s) \|_{\ell^2(H)}^2 \big) ds < \infty, \quad t \in \bbr_+
\end{align}
and $\bbp$-almost surely
\begin{align}\label{SPDE-integral-form-var}
Y_{t \wedge \tau} = y_0 + \text{{\rm ($K$-)}} \int_0^{t \wedge \tau} L(Y_s) ds + \text{{\rm ($H$-)}} \int_0^{t \wedge \tau} A(Y_s) dW_s, \quad t \in \bbr_+.
\end{align}
The stopping time $\tau$ is also called the \emph{lifetime} of $Y$.
\end{enumerate}
\end{definition}

\begin{remark}
As the notation indicates, the stochastic integral
\begin{align*}
\text{{\rm ($K$-)}} \int_0^{t \wedge \tau} L(Y_s) ds
\end{align*}
appearing in (\ref{SPDE-integral-form-var}) is a pathwise Bochner integral in the Banach space $(K,\| \cdot \|_K)$, and the stochastic integral
\begin{align*}
\text{{\rm ($H$-)}} \int_0^{t \wedge \tau} A(Y_s) dW_s
\end{align*}
appearing in (\ref{SPDE-integral-form-var}) is an It\^{o} integral in the separable Hilbert space $(H,\| \cdot \|_H)$.
\end{remark}

\begin{remark}
If there is no ambiguity, we will simply call $Y$ a local martingale solution to the $(G,H,K)$-variational SPDE (\ref{SPDE}) with $Y_0 = y_0$.
\end{remark}

\begin{remark}\label{rem-cont-var}
Note that every local martingale solution $Y$ to the $(G,H,K)$-variational SPDE (\ref{SPDE}) with lifetime $\tau$ is a $G$-valued process, and that by (\ref{SPDE-integral-form-var}) the sample paths of $Y^{\tau}$ are continuous with respect to $\| \cdot \|_K$.
\end{remark}

\begin{definition}\label{def-local-invariance-var}
A subset $\calm \subset \caln$ is called \emph{locally invariant} for the $(G,H,K)$-variational SPDE (\ref{SPDE}) if for each $y_0 \in \calm$ there exists a local martingale solution $Y$ to the $(G,H,K)$-variational SPDE (\ref{SPDE}) with $Y_0 = y_0$ and lifetime $\tau > 0$ such that $Y^{\tau} \in \calm$ up to an evanescent set.
\end{definition}

\section{The situation where $K$ is a separable Hilbert space}\label{sec-K-separable}

In this section we investigate invariance of finite dimensional submanifolds in the situation where $K$ is a separable Hilbert space. More precisely, let $G$ be a Banach space and let $H,K$ be separable Hilbert spaces such that $(G,H,K)$ is a triplet of continuously embedded spaces. Furthermore, let $L : G \to K$ and $A : G \to \ell^2(H)$ be measurable mappings. Note that the present situation applies to the variational SPDEs considered in \cite{Rozovskii}.

\begin{remark}
Note that the spaces $(\ell^2(H),\ell^2(K))$ are also continuously embedded. Hence, we may and will consider $A$ also as a measurable mapping $A : G \to \ell^2(K)$.
\end{remark}

\begin{proposition}\label{prop-solutions-1}
Let $y_0 \in G$ be arbitrary, and let $Y$ be a local martingale solution to the $(G,H,K)$-variational SPDE (\ref{SPDE}) with lifetime $\tau$ and $Y_0 = y_0$. Then $Y$ is also a local martingale solution to the $(G,K)$-embedded SPDE (\ref{SPDE}) with lifetime $\tau$ and $Y_0 = y_0$.
\end{proposition}

\begin{proof}
Taking into account condition (\ref{int-cond-var}), by Lemma \ref{lemma-Wiener-int-G-H} we have $\bbp$-almost surely
\begin{align}\label{int-cond-1}
\int_0^{t \wedge \tau} \big( \| L(Y_s) \|_K + \| A(Y_s) \|_{\ell^2(K)}^2 \big) ds < \infty, \quad t \in \bbr_+
\end{align}
as well as
\begin{align}\label{H-K-int-coincides}
\text{{\rm ($H$-)}} \int_0^{t \wedge \tau} A(Y_s) d W_s = \text{{\rm ($K$-)}} \int_0^{t \wedge \tau} A(Y_s) d W_s, \quad t \in \bbr_+.
\end{align}
Hence, by (\ref{SPDE-integral-form-var}) we obtain $\bbp$-almost surely
\begin{align*}
Y_{t \wedge \tau} &= y_0 + \text{{\rm ($K$-)}} \int_0^{t \wedge \tau} L(Y_s) ds + \text{{\rm ($H$-)}} \int_0^{t \wedge \tau} A(Y_s) dW_s
\\ &= y_0 + \text{{\rm ($K$-)}} \int_0^{t \wedge \tau} L(Y_s) ds + \text{{\rm ($K$-)}} \int_0^{t \wedge \tau} A(Y_s) dW_s, \quad t \in \bbr_+,
\end{align*}
showing that $Y$ is a local martingale solution to the $(G,K)$-embedded SPDE (\ref{SPDE}) with lifetime $\tau$ and $Y_0 = y_0$.
\end{proof}

\begin{proposition}\label{prop-solutions-2}
Let $\calm \subset G$ be a subset such that $A|_{\calm} : (\calm,\| \cdot \|_G) \to \ell^2(H)$ is continuous. Let $y_0 \in \calm$ be arbitrary, and let $Y$ be a local martingale solution to the $(G,K)$-embedded SPDE (\ref{SPDE}) with lifetime $\tau$ and $Y_0 = y_0$ such that $Y^{\tau} \in \calm$ and the sample paths of $Y^{\tau}$ are continuous with respect to $\| \cdot \|_G$. Then $Y$ is also a local martingale solution to the $(G,H,K)$-variational SPDE (\ref{SPDE}) with lifetime $\tau$ and $Y_0 = y_0$.
\end{proposition}

\begin{proof}
Note that for each $\omega \in \Omega$ the image
\begin{align*}
\{ L(Y_t^{\tau(\omega)}(\omega)) : t \in \bbr_+ \} \subset K
\end{align*}
is a separable subset of $(K,\| \cdot \|_K)$, because $K$ is a separable Hilbert space. Since $A|_{\calm} : (\calm,\| \cdot \|_G) \to \ell^2(H)$ is continuous and the sample paths of $Y^{\tau}$ are continuous with respect to $\| \cdot \|_G$, it follows that the sample paths of the composition $A(Y^{\tau})$ are continuous with respect to $\| \cdot \|_{\ell^2(H)}$, which implies
\begin{align*}
\int_0^{t \wedge \tau} \| A(Y_s) \|_{\ell^2(H)}^2 ds < \infty, \quad t \in \bbr_+.
\end{align*}
Therefore, condition (\ref{int-cond-1}) implies (\ref{int-cond-var}), and hence, by Lemma \ref{lemma-Wiener-int-G-H} we have the identity (\ref{H-K-int-coincides}). Consequently, since $Y$ be a local martingale solution to the $(G,K)$-embedded SPDE (\ref{SPDE}) with lifetime $\tau$ and $Y_0 = y_0$, we obtain $\bbp$-almost surely
\begin{align*}
Y_{t \wedge \tau} &= y_0 + \text{{\rm ($K$-)}} \int_0^{t \wedge \tau} L(Y_s) ds + \text{{\rm ($K$-)}} \int_0^{t \wedge \tau} A(Y_s) dW_s
\\ &= y_0 + \text{{\rm ($K$-)}} \int_0^{t \wedge \tau} L(Y_s) ds + \text{{\rm ($H$-)}} \int_0^{t \wedge \tau} A(Y_s) dW_s, \quad t \in \bbr_+,
\end{align*}
showing that $Y$ is a local martingale solution to the $(G,H,K)$-variational SPDE (\ref{SPDE}) with lifetime $\tau$ and $Y_0 = y_0$.
\end{proof}

\begin{theorem}\label{thm-K-separable}
Let $\calm$ be a finite dimensional $(G,K)$-submanifold of class $C^2$. Suppose that $L|_{\calm} : (\calm,\| \cdot \|_G) \to K$ and $A|_{\calm} : (\calm,\| \cdot \|_G) \to \ell^2(H)$ are continuous. Then the following statements are equivalent:
\begin{enumerate}
\item[(i)] The submanifold $\calm$ is locally invariant for the $(G,K)$-embedded SPDE (\ref{SPDE}).

\item[(ii)] The submanifold $\calm$ is locally invariant for the $(G,H,K)$-variational SPDE (\ref{SPDE}).

\item[(iii)] We have (\ref{tang-A}) and (\ref{tang-L}).
\end{enumerate}
In particular, if the equivalent conditions (i)--(iii) are fulfilled, then for each $y_0 \in \calm$ there exists a local martingale solution $Y$ to the $(G,K)$-embedded SPDE (\ref{SPDE}) and to the $(G,H,K)$-variational SPDE (\ref{SPDE}) with $Y_0 = y_0$.
\end{theorem}

\begin{proof}
(i) $\Rightarrow$ (ii): Let $y_0 \in \calm$ be arbitrary. Then there exists a local martingale solution $Y$ to the $(G,K)$-embedded SPDE (\ref{SPDE}) with $Y_0 = y_0$ and lifetime $\tau$ such that $Y^{\tau} \in \calm$ up to an evanescent set. By Remark \ref{rem-cont-embedded} the sample paths of $Y^{\tau}$ are continuous with respect to $\| \cdot \|_K$, and hence, by Proposition \ref{prop-embedding} it follows that the sample paths of $Y^{\tau}$ are even continuous with respect to $\| \cdot \|_G$. Therefore, by Proposition \ref{prop-solutions-2} the process $Y$ is also a local martingale solution to the $(G,H,K)$-variational SPDE (\ref{SPDE}) with lifetime $\tau$ and $Y_0 = y_0$.

\noindent(ii) $\Rightarrow$ (i): This is a consequence of Proposition \ref{prop-solutions-1}.

\noindent(i) $\Leftrightarrow$ (iii): Since $(H,K)$ are continuously embedded Hilbert spaces, it follows that the spaces of sequences $(\ell^2(H),\ell^2(K))$ are also continuously embedded. Hence, we may regard $A|_{\calm}$ as a continuous mapping $A|_{\calm} : (\calm, \| \cdot \|_G) \to \ell^2(K)$, and thus applying Theorem \ref{thm-SPDE} proves the stated equivalence.
\end{proof}

\begin{corollary}
Let $\calm$ is a finite dimensional $(G,H,K)$-submanifold of class $C^2$. Suppose that for each $j \in \bbn$ we have $A^j \in C(G;H)$ with an extension $A^j \in C^1(H;K)$, and that for each $y \in \calm$ the series $\sum_{j=1}^{\infty} D A^j(y) A^j(y)$ converges in $K$. Then the following statements are equivalent:
\begin{enumerate}
\item[(i)] The submanifold $\calm$ is locally invariant for the $(G,K)$-embedded SPDE (\ref{SPDE}).

\item[(ii)] The submanifold $\calm$ is locally invariant for the $(G,H,K)$-variational SPDE (\ref{SPDE}).

\item[(iii)] We have (\ref{tang-A}) and
\begin{align*}
L|_{\calm} - \frac{1}{2} \sum_{j=1}^{\infty} D A^j \cdot A^j|_{\calm} \in \Gamma(T \calm).
\end{align*}
\end{enumerate}
In particular, if the equivalent conditions (i)--(iii) are fulfilled, then for each $y_0 \in \calm$ there exists a local martingale solution $Y$ to the $(G,K)$-embedded SPDE (\ref{SPDE}) and to the $(G,H,K)$-variational SPDE (\ref{SPDE}) with $Y_0 = y_0$.
\end{corollary}

\begin{proof}
Note that $\calm$ is also a finite dimensional $(G,K)$-submanifold of class $C^2$. Let $\phi : V \to U \cap \calm$ be a local parametrization of $\calm$, and let $j \in \bbn$ be arbitrary. According to the decomposition (3.2) from \cite[Prop. 3.25]{BT-preprint} we have
\begin{align*}
D A^j(y) A^j(y) = D \phi(x) ( D a^j(x) a^j(x) ) + D^2 \phi(x) ( a^j(x), a^j(x) ), \quad y \in U \cap \calm,
\end{align*}
where $x := \phi^{-1}(y) \in V$, and the mapping $a^j \in C^1(V;\bbr^m)$ with $m := \dim \calm$ is given by
\begin{align*}
a^j(x) := D \phi(x)^{-1} A^j(y), \quad x \in V,
\end{align*}
where $y := \phi(x) \in U \cap \calm$. Thus, recalling that local representatives of $[A,B]_{\calm}$ for two vector fields $A,B \in \Gamma(T \calm)$ are given by (\ref{bracket}), applying Theorem \ref{thm-K-separable} concludes the proof.
\end{proof}

\section{The general situation}\label{sec-general}

In this section we investigate invariance of finite dimensional submanifolds in the general situation. Let $G,K$ be Banach spaces and let $H$ be a separable Hilbert space such that $(G,H,K)$ is a triplet of continuously embedded spaces. Furthermore, let $L : G \to K$ and $A : G \to \ell^2(H)$ be measurable mappings. Note that this covers the situation with a Gelfand triplet $(G,H,K) = (V,H,V^*)$ for some reflexive Banach space $V$; see, for example \cite{Liu-Roeckner}.

\begin{lemma}\label{lemma-L-restr}
Let $\calm \subset G$ be a subset such that $L(\calm) \subset H$. Then the restriction $L|_{\calm} : \calm \to H$ is $\calb(\calm)$-$\calb(H)$-measurable.
\end{lemma}

\begin{proof}
By Lemma \ref{lemma-Kuratowski} we have $H \in \calb(K)$ and $\calb(H) = \calb(K)_H$. Therefore, for each $B \in \calb(K)$ we have $B \cap H \in \calb(K)$, and hence
\begin{align*}
L|_{\calm}^{-1}(B \cap H) = L^{-1}(B \cap H) \cap \calm \in \calb(G)_{\calm} = \calb(\calm),
\end{align*}
which completes the proof.
\end{proof}

\begin{proposition}\label{prop-solutions-G-H-1}
Let $\calm \subset G$ be a subset such that $L(\calm) \subset H$ and $L|_{\calm} : (\calm,\| \cdot \|_G) \to (H,\| \cdot \|_H)$ is continuous. Let $y_0 \in \calm$ be arbitrary, and let $Y$ be a local martingale solution to the $(G,H,K)$-variational SPDE (\ref{SPDE}) with lifetime $\tau$ and $Y_0 = y_0$ such that $Y^{\tau} \in \calm$ and the sample paths of $Y^{\tau}$ are continuous with respect to $\| \cdot \|_G$. Then $Y$ is also a local martingale solution to the $(G,H)$-embedded SPDE (\ref{SPDE}) with lifetime $\tau$ and $Y_0 = y_0$.
\end{proposition}

\begin{proof}
Since $L|_{\calm} : (\calm,\| \cdot \|_G) \to (H,\| \cdot \|_H)$ is continuous and the sample paths of $Y^{\tau}$ are continuous with respect to $\| \cdot \|_G$, it follows that the sample paths of the composition $L(Y^{\tau})$ are continuous with respect to $\| \cdot \|_H$. Therefore, we have $\bbp$-almost surely
\begin{align}\label{L-H-integrable}
\int_0^{t \wedge \tau} \| L(Y_s) \|_H ds < \infty, \quad t \in \bbr_+,
\end{align}
and hence condition (\ref{int-cond-var}) implies (\ref{SPDE-int-cond}). Furthermore, taking into account Lemma \ref{lemma-Leb-int-G-H} we have $\bbp$-almost surely
\begin{align}\label{H-K-int-Bochner-coincides}
\text{{\rm ($H$-)}} \int_0^{t \wedge \tau} L(Y_s) ds = \text{{\rm ($K$-)}} \int_0^{t \wedge \tau} L(Y_s) ds, \quad t \in \bbr_+.
\end{align}
Hence, the identity (\ref{SPDE-integral-form-var}) implies that $\bbp$-almost surely
\begin{align*}
Y_{t \wedge \tau} &= y_0 + \text{{\rm ($K$-)}} \int_0^{t \wedge \tau} L(Y_s) ds + \text{{\rm ($H$-)}} \int_0^{t \wedge \tau} A(Y_s) dW_s
\\ &= y_0 + \text{{\rm ($H$-)}} \int_0^{t \wedge \tau} L(Y_s) ds + \text{{\rm ($H$-)}} \int_0^{t \wedge \tau} A(Y_s) dW_s, \quad t \in \bbr_+,
\end{align*}
showing that $Y$ is a local martingale solution to the $(G,H)$-embedded SPDE (\ref{SPDE}) with lifetime $\tau$ and $Y_0 = y_0$.
\end{proof}

\begin{proposition}\label{prop-solutions-G-H-2}
Let $\calm \subset G$ be a subset such that $L(\calm) \subset H$. Let $y_0 \in \calm$ be arbitrary, and let $Y$ be a local martingale solution to the $(G,H)$-embedded SPDE (\ref{SPDE}) with lifetime $\tau$ and $Y_0 = y_0$ such that $Y^{\tau} \in \calm$. Then $Y$ is also a local martingale solution to the $(G,H,K)$-variational SPDE (\ref{SPDE}) with lifetime $\tau$ and $Y_0 = y_0$.
\end{proposition}

\begin{proof}
Since $Y$ is predictable, the stopped process $Y^{\tau}$ is also predictable, and thus $\calf \otimes \calb(\bbr_+)$-measurable. Hence, for each $\omega \in \Omega$ the sample path $t \mapsto Y_t^{\tau(\omega)}(\omega)$ is $\calb(\bbr_+)$-$\calb(\calm)$-measurable. Therefore, by Lemma \ref{lemma-L-restr} for each $\omega \in \Omega$ the mapping $t \mapsto L(Y_t^{\tau(\omega)}(\omega))$ is $\calb(\bbr_+)$-$\calb(H)$-measurable. Taking into account condition (\ref{SPDE-int-cond}), by Lemma \ref{lemma-Leb-int-G-H} we obtain $\bbp$-almost surely
\begin{align}\label{L-int}
\int_0^{t \wedge \tau} \| L(Y_s) \|_K ds < \infty, \quad t \in \bbr_+
\end{align}
as well as (\ref{H-K-int-Bochner-coincides}). Now, condition (\ref{SPDE-int-cond}) and (\ref{L-int}) imply condition (\ref{int-cond-var}), and the identity (\ref{SPDE-integral-form}) implies that $\bbp$-almost surely
\begin{align*}
Y_{t \wedge \tau} &= y_0 + \text{{\rm ($H$-)}} \int_0^{t \wedge \tau} L(Y_s) ds + \text{{\rm ($H$-)}} \int_0^{t \wedge \tau} A(Y_s) dW_s
\\ &= y_0 + \text{{\rm ($K$-)}} \int_0^{t \wedge \tau} L(Y_s) ds + \text{{\rm ($H$-)}} \int_0^{t \wedge \tau} A(Y_s) dW_s, \quad t \in \bbr_+,
\end{align*}
showing that $Y$ is a local martingale solution to the $(G,H,K)$-variational SPDE (\ref{SPDE}) with lifetime $\tau$ and $Y_0 = y_0$.
\end{proof}

\begin{theorem}\label{thm-general}
Let $\calm$ be a finite dimensional $(G,K)$-submanifold of class $C^2$. Suppose that $L(\calm) \subset H$ and that $L|_{\calm} : (\calm,\| \cdot \|_G) \to (H,\| \cdot \|_H)$ and $A|_{\calm} : (\calm,\| \cdot \|_G) \to \ell^2(H)$ are continuous. Then the following statements are equivalent:
\begin{enumerate}
\item[(i)] The submanifold $\calm$ is locally invariant for the $(G,H)$-embedded SPDE (\ref{SPDE}).

\item[(ii)] The submanifold $\calm$ is locally invariant for the $(G,H,K)$-variational SPDE (\ref{SPDE}).

\item[(iii)] We have (\ref{tang-A}) and (\ref{tang-L}).
\end{enumerate}
In particular, if the equivalent conditions (i)--(iii) are fulfilled, then for each $y_0 \in \calm$ there exists a local martingale solution $Y$ to the $(G,H)$-embedded SPDE (\ref{SPDE}) and to the $(G,H,K)$-variational SPDE (\ref{SPDE}) with $Y_0 = y_0$.
\end{theorem}

\begin{proof}
(i) $\Rightarrow$ (ii): This is a consequence of Proposition \ref{prop-solutions-G-H-2}.

\noindent(ii) $\Rightarrow$ (i): Let $y_0 \in \calm$ be arbitrary. Then there exists a local martingale solution $Y$ to the $(G,H,K)$-variational SPDE (\ref{SPDE}) with $Y_0 = y_0$ and lifetime $\tau$ such that $Y^{\tau} \in \calm$ up to an evanescent set. By Remark \ref{rem-cont-var} the sample paths of $Y^{\tau}$ are continuous with respect to $\| \cdot \|_K$, and hence, by Proposition \ref{prop-embedding} it follows that the sample paths of $Y^{\tau}$ are even continuous with respect to $\| \cdot \|_G$. Therefore, by Proposition \ref{prop-solutions-G-H-1} the process $Y$ is also a local martingale solution to the $(G,H)$-embedded SPDE (\ref{SPDE}) with lifetime $\tau$ and $Y_0 = y_0$.

\noindent(i) $\Leftrightarrow$ (iii): Since $(H,K)$ are continuously embedded spaces, the embedding operator $\Id : (H,\| \cdot \|_H) \to (K,\| \cdot \|_K)$ is continuous, and hence, by Proposition \ref{prop-embedding} the submanifold $\calm$ is also a $(G,H)$-submanifold of class $C^2$. Therefore, the stated equivalence is a consequence of Theorem \ref{thm-SPDE}.
\end{proof}

\section{Invariant submanifolds in Hermite Sobolev spaces}\label{sec-HS}

In this section we present our first application. Namely, we use our findings from Section \ref{sec-K-separable} in order to construct examples of invariant submanifolds in Hermite Sobolev spaces; see \cite[App. A]{BT-preprint} for the required background about Hermite Sobolev spaces. Let $p \in \bbr$ be arbitrary and define the Hermite Sobolev spaces
\begin{align*}
G := \cals_{p+1}(\bbr^d), \quad H := \cals_{p+\frac{1}{2}}(\bbr^d) \quad \text{and} \quad K := \cals_p(\bbr^d).
\end{align*}
Then $(G,H,K)$ is a triplet of continuously embedded separable Hilbert spaces, and hence we are in the framework of Section \ref{sec-K-separable}. Let
\begin{align*}
b \in \cals_{-(p+1)}(\bbr^d;\bbr^d) \quad \text{and} \quad \sigma \in \ell^2(\cals_{-(p+1)}(\bbr^d;\bbr^d))
\end{align*}
be given mappings. We define the coefficients $L : G \to K$ and $A^j : G \to H$ for $j \in \bbn$ as
\begin{align*}
L(y) &:= \frac{1}{2} \sum_{i,j=1}^d ( \langle \sigma,y \rangle \langle \sigma,y \rangle^{\top} )_{ij} \partial_{ij}^2 y - \sum_{i=1}^d \langle b_i,y \rangle \partial_i y,
\\ A^j(y) &:= - \sum_{i=1}^d \langle \sigma_{i}^j,y \rangle \partial_i y, \quad j \in \bbn.
\end{align*}
This provides well-defined continuous mappings $L : G \to K$ and $A : G \to \ell^2(H)$ for the SPDE (\ref{SPDE}); see \cite[Sec. 5.3]{BT} for further details. Such SPDEs were recently studied in \cite{Rajeev, Rajeev-2019}. We will call them It\^{o} type SPDEs due to their connection to finite dimensional SDEs; see \cite[Sec. 6]{BT}. As an immediate consequence of Theorem \ref{thm-K-separable} we obtain the following result.

\begin{proposition}\label{prop-K-separable}
Let $\calm$ be a finite dimensional $(G,K)$-submanifold of class $C^2$. Then the following statements are equivalent:
\begin{enumerate}
\item[(i)] The submanifold $\calm$ is locally invariant for the $(G,K)$-embedded It\^{o} type SPDE (\ref{SPDE}).

\item[(ii)] The submanifold $\calm$ is locally invariant for the $(G,H,K)$-variational It\^{o} type SPDE (\ref{SPDE}).
\end{enumerate}
\end{proposition}

In \cite[Sec. 5.3]{BT} there are several results and examples concerning locally invariant submanifolds for the $(G,K)$-embedded It\^{o} type SPDE (\ref{SPDE}); in particular for submanifolds generated by the translation group $(\tau_x)_{x \in \bbr^d}$, which are of the form
\begin{align*}
\calm = \{ \tau_x \Phi : x \in \caln \}
\end{align*}
for some $\Phi \in G$ and a $C^2$-submanifold $\caln$ of $\bbr^d$. As a consequence of Proposition \ref{prop-K-separable}, all the aforementioned findings from \cite[Sec. 5.3]{BT} transfer to the $(G,H,K)$-variational It\^{o} type SPDE (\ref{SPDE}).

\section{Linear submanifolds for the stochastic $p$-Laplace equation}\label{sec-Laplace}

In this section we present our second application. Namely, we use our findings from Section \ref{sec-general} in order to characterize linear submanifolds for the stochastic $p$-Laplace equation (cf. \cite[Example 1.2.6]{Liu-Roeckner})
\begin{align}\label{SPDE-Laplace}
\left\{
\begin{array}{rcl}
dY_t & = & {\rm div} \, \big( | \nabla Y_t |^{p-2} \nabla Y_t \big) dt + A(Y_t) dW_t
\\ Y_0 & = & y_0.
\end{array}
\right.
\end{align}
Let us start with a more general situation. Namely, consider a Gelfand triplet $(V,H,V^*)$. We assume that the drift $L : V \to V^*$ in the SPDE (\ref{SPDE}) is a continuous linear operator $L \in L(V,V^*)$. Furthermore, let $A : V \to \ell^2(H)$ be a measurable mapping. Let $v_1,\ldots,v_m \in V$ be eigenvectors or elements from the kernel of $L$ for some $m \in \bbn$; that is, we have
\begin{align*}
L v_i = \lambda_i v_i, \quad i=1,\ldots,m
\end{align*}
with $\lambda_1,\ldots,\lambda_m \in \bbr$. We consider the linear submanifold
\begin{align*}
\calm := \lin \{ v_1,\ldots,v_m \}.
\end{align*}

\begin{proposition}\label{prop-linear-manifold}
Suppose that $A|_{\calm} : (\calm,\| \cdot \|_G) \to \ell^2(H)$ is continuous. Then the following statements are equivalent:
\begin{enumerate}
\item[(i)] The submanifold $\calm$ is locally invariant for the $(V,H)$-embedded SPDE (\ref{SPDE}).

\item[(ii)] The submanifold $\calm$ is locally invariant for the $(V,H,V^*)$-variational SPDE (\ref{SPDE}).

\item[(iii)] We have (\ref{tang-A}).
\end{enumerate}
\end{proposition}

\begin{proof}
Note that $L(\calm) \subset V$ and that $L|_{\calm} : \calm \to V$ is continuous, showing that the assumptions of Theorem \ref{thm-general} are more than satisfied. Hence, together with \cite[Cor. 3.9]{BT} this completes the proof.
\end{proof}

Now we consider the stochastic $p$-Laplace equation (\ref{SPDE-Laplace}). Let us fix some $p \in [2,\infty)$. The space $V$ is given by $V := H_0^{1,p}(\Lambda)$ with an open subset $\Lambda \subset \bbr^d$; the so-called Sobolev space of order $1$ in $L^p(\Lambda)$; see \cite[p. 78]{Liu-Roeckner}. The drift $L$ in the SPDE (\ref{SPDE}) is the $p$-Laplacian $L : V \to V^*$ given by
\begin{align*}
L(y) = {\rm div} \, \big( | \nabla y |^{p-2} \nabla y \big), \quad y \in V.
\end{align*}
The $p$-Laplacian $L : V \to V^*$ is indeed a continuous linear operator $L \in L(V,V^*)$; see equation (4.14) in \cite[p. 81]{Liu-Roeckner}. Therefore, Proposition \ref{prop-linear-manifold} applies to every finite dimensional submanifold $\calm$ generated by eigenvectors or elements from the kernel of $L$.

\begin{appendix}

\section{Continuously embedded spaces}\label{app-embedding}

In this appendix we provide the required auxiliary results about continuously embedded spaces.

\begin{definition}\label{def-embedding}
Let $H$ and $K$ be two normed spaces. Then we call $(H,K)$ \emph{continuously embedded normed spaces} (or \emph{normed spaces with continuous embedding}) if the following conditions are fulfilled:
\begin{enumerate}
\item We have $H \subset K$ as sets.

\item The embedding operator $\Id : (H, \| \cdot \|_H) \to (K, \| \cdot \|_K)$ is continuous; that is, there is a constant $C > 0$ such that
\begin{align}\label{embedding-inequ}
\| x \|_K \leq C \| x \|_H \quad \text{for all $x \in H$.}
\end{align}
\end{enumerate}
\end{definition}

\begin{lemma}\label{lemma-Kuratowski}
Let $(H,K)$ be two continuously embedded Banach spaces. Then we have $H \in \calb(K)$ and $\calb(H) = \calb(K)_H$, where $\calb(K)_H$ denotes the trace $\sigma$-algebra
\begin{align*}
\calb(K)_H = \{ B \cap H : B \in \calb(K) \}.
\end{align*}
\end{lemma}

\begin{proof}
This is a consequence of Kuratowski's theorem; see, for example \cite[Thm. I.3.9]{Parthasarathy}.
\end{proof}

\begin{lemma}\label{lemma-Leb-int-G-H}
Let $(H,K)$ be a pair of continuously embedded Banach spaces, and let $(E,\cale,\mu)$ be a finite measure space. Let $f : E \to H$ be an $\cale$-$\calb(H)$-measurable mapping such that the image $f(E)$ is a separable subset of $(H,\| \cdot \|_H)$ and
\begin{align}\label{f-H-integrable}
\int_E \| f \|_H \, d \mu < \infty.
\end{align}
Then the following statements are true:
\begin{enumerate}
\item The mapping $f$ is also $\cale$-$\calb(K)$-measurable, the image $f(E)$ is a separable subset of $(K,\| \cdot \|_K)$, and we have
\begin{align}\label{f-K-integrable}
\int_E \| f \|_K \, d \mu < \infty.
\end{align}
\item We have
\begin{align}\label{H-K-coincide}
\text{{\rm ($H$-)}} \int_E f \, d \mu = \text{{\rm ($K$-)}} \int_E f \, d \mu.
\end{align}
\end{enumerate}
\end{lemma}

\begin{proof}
For each $B \in \calb(K)$ we have
\begin{align*}
f^{-1}(B \cap H) = f^{-1}(B) \cap f^{-1}(H) = f^{-1}(B) \cap E = f^{-1}(B).
\end{align*}
Furthermore, by Lemma \ref{lemma-Kuratowski} we have $\calb(K)_H = \calb(H)$, and hence
\begin{align*}
f^{-1}(\calb(K)) = f^{-1}(\calb(K)_H) = f^{-1}(\calb(H)) \subset \cale.
\end{align*}
There is a countable subset $D \subset f(E)$ which is dense in $f(E)$ with respect to $\| \cdot \|_H$. Let $x \in f(E)$ be arbitrary. Then there exists a sequence $(x_n)_{n \in \bbn} \subset D$ such that $\| x_n - x \|_H \to 0$, which implies $\| x_n - x \|_K \to 0$ due to (\ref{embedding-inequ}). Furthermore, by (\ref{embedding-inequ}) and (\ref{f-H-integrable}) we have (\ref{f-K-integrable}), proving the first statement. For the proof of the second statement we proceed in two steps. Suppose first that the mapping $f$ is simple; that is, for some $n \in \bbn$ we have
\begin{align*}
f = \sum_{i=1}^n c_i \bbI_{B_i}
\end{align*}
with $c_i \in H$ and $B_i \in \cale$ for each $i=1,\ldots,n$. Then we have
\begin{align}\label{integral-simple}
\text{{\rm ($H$-)}} \int_E f \, d \mu = \sum_{i=1}^n c_i \cdot \mu(B_i) = \text{{\rm ($K$-)}} \int_E f \, d \mu.
\end{align}
Now, we consider the general situation. According to \cite[Lemma I.1.3]{Da_Prato} there exists a sequence $(f_n)_{n \in \bbn}$ of simple functions $f_n : E \to H$ such that $\| f_n(x) - f(x) \|_H \downarrow 0$ for each $x \in E$. Hence, for each $n \in \bbn$ and each $x \in E$ we have
\begin{align*}
\| f_n(x) - f(x) \|_H \leq \| f_1(x) - f(x) \|_H \leq \| f(x) \|_H + \| f_1(x) \|_H.
\end{align*}
By Lebesgue's dominated convergence theorem we obtain
\begin{align}\label{fn-f-H}
\int_E \| f_n - f \|_H \, d\mu \to 0,
\end{align}
and hence
\begin{align}\label{fn-H}
\int_E f_n \, d\mu \to \text{{\rm ($H$-)}} \int_E f \, d \mu \quad \text{in $(H,\| \cdot \|_H)$,}
\end{align}
where $\int_E f_n \, d\mu$ is defined according to (\ref{integral-simple}). By (\ref{embedding-inequ}) and (\ref{fn-f-H}) we have
\begin{align*}
\int_E \| f_n - f \|_K \, d\mu \to 0,
\end{align*}
and hence
\begin{align}\label{fn-K}
\int_E f_n \, d\mu \to \text{{\rm ($K$-)}} \int_E f \, d \mu \quad \text{in $(K,\| \cdot \|_K)$.}
\end{align}
Moreover, by (\ref{embedding-inequ}) and (\ref{fn-H}) we have
\begin{align}\label{fn-HK}
\int_E f_n \, d\mu \to \text{{\rm ($H$-)}} \int_E f \, d \mu \quad \text{in $(K,\| \cdot \|_K)$.}
\end{align}
Combining (\ref{fn-K}) and (\ref{fn-HK}) we arrive at (\ref{H-K-coincide}).
\end{proof}

With a similar proof, we obtain the following auxiliary result concerning the It\^{o} integral.

\begin{lemma}\label{lemma-Wiener-int-G-H}
Let $(H,K)$ be a pair of continuously embedded separable Hilbert spaces, and let $W$ be a standard $\bbr^{\infty}$-Wiener process on some stochastic basis $\bbb$. Let $A$ be an $\ell^2(H)$-valued $\calp$-$\calb(\ell^2(H))$-measurable process such that $\bbp$-almost surely
\begin{align*}
\int_0^t \| A_s \|_{\ell^2(H)}^2 \, ds < \infty, \quad t \in \bbr_+.
\end{align*}
Then the following statements are true:
\begin{enumerate}
\item $A$ is also an $\ell^2(K)$-valued $\calp$-$\calb(\ell^2(K))$-measurable process and we have $\bbp$-almost surely
\begin{align*}
\int_0^t \| A_s \|_{\ell^2(K)}^2 \, ds < \infty, \quad t \in \bbr_+.
\end{align*}
\item We have $\bbp$-almost surely
\begin{align*}
\text{{\rm ($H$-)}} \int_0^t A_s dW_s = \text{{\rm ($K$-)}} \int_0^t A_s dW_s, \quad t \in \bbr_+.
\end{align*}
\end{enumerate}
\end{lemma}

\end{appendix}

\end{document}